\documentclass[12pt]{article}
\usepackage{amsmath}
\usepackage{float}
\usepackage{amsthm}
\usepackage{amssymb}
\usepackage[utf8x]{inputenc}
\usepackage{caption}
\usepackage{graphics}
\usepackage{enumerate} 
\usepackage{algorithmic}
\usepackage[Algorithm,ruled]{algorithm}
\usepackage{titlesec}
\usepackage{sectsty}
\titleformat*{\section}{\LARGE\bfseries}
\titleformat*{\subsection}{\Large\bfseries}
\titleformat*{\subsubsection}{\large\bfseries}
\sectionfont{\Huge}
\subsectionfont{\large}
\subsubsectionfont{\normalsize}
\usepackage{multicol}
\usepackage{lmodern}
\usepackage{marvosym} 
\usepackage{lipsum}
\usepackage{mwe}
\usepackage{caption}
\usepackage{subcaption} 
\newtheoremstyle{case}{}{}{}{}{}{:}{ }{}
\theoremstyle{case}

\newcommand{\be}{\begin{equation}}
\newcommand{\ee}{\end{equation}}
\newcommand{\ben}{\begin{eqnarray*}}
\newcommand{\een}{\end{eqnarray*}}
\newtheorem{examp}{\sc example}
\newtheorem{remk}{\sc remark}
\newtheorem{corol}{\sc corollary}
\newtheorem{lemma}{\sc lemma}
\newtheorem{theorem}{\sc theorem}
\newtheorem{defn}{\sc definition}
\newcommand{\bt}{\begin{theorem}}
\newcommand{\et}{\end{theorem}}
\newcommand{\bl}{\begin{lemma}}
\newcommand{\el}{\end{lemma}}
\newcommand{\bed}{\begin{defn}}
\newcommand{\eed}{\end{defn}}
\newcommand{\brem}{\begin{remk}}
\newcommand{\erem}{\end{remk}}
\newcommand{\bex}{\begin{examp}}
\newcommand{\eex}{\end{examp}}
\newcommand{\bcl}{\begin{corol}}
\newcommand{\ecl}{\end{corol}}

\newcommand{\NI}{\noindent}

\theoremstyle{definition}
\theoremstyle{remark}

\numberwithin{equation}{section}
\numberwithin{theorem}{section}
\numberwithin{lemma}{section}

\begin{document}

\title{\large\bf\sc On some properties of set-valued tensor complementarity problem}

\author{R. Deb$^{a,1}$ and A. K. Das$^{b,2}$\\
\emph{\small $^{a}$Jadavpur University, Kolkata , 700 032, India.}\\	
\emph{\small $^{b}$Indian Statistical Institute, 203 B. T.
	Road, Kolkata, 700 108, India.}\\
\emph{\small $^{1}$Email: rony.knc.ju@gmail.com}\\
\emph{\small $^{2}$Email: akdas@isical.ac.in}\\
}

\date{}

\maketitle

\begin{abstract}
\NI In this paper, we introduce set-valued tensor complementarity problem where the elements of the involved tensors are defined based on a set-valued mapping. We study several properties of the solution set under the framework of set-valued mapping. We provide the necessary and sufficient conditions for the zero solution of a set-valued tensor complementarity problem. We introduce limit $R_0$-property for the set of tensors and establish a connection between limit $R_0$-property and the level boundedness of the merit function of the corresponding set-valued tensor complementarity problem.\\


\end{abstract}
\footnotetext[1]{Corresponding author}

\section{Introduction}
Song and Qi introduced the tensor complementarity problem where the associated functions are the special polynomials constructed with the help of tensors. For details see \cite{song2017properties} and \cite{song2015properties}. Huang and Qi \cite{huang2017formulating} showed that the multilinear non-cooperative game can be formulated as a tensor complementarity problem, which establishes a connection between these two classes of problems.

\NI Given an $m$th order $n$ dimensional tensor $\mathcal{B}\in T_{m,n}$ and an $n$ dimensional vector $p\in \mathbb{R}^n,$ the tensor complementarity problem is to find $v \in \mathbb{R}^n$ such that
\begin{equation*}\label{ Tensor Complementarity problem}
		p + \mathcal{B}v^{m-1}\in \mathbb{R}^n_+, \;\;\;\;  v\in \mathbb{R}^n_+, \;\;\;\; \mbox{and}\;\;\;\; v^{T}(p + \mathcal{B}v^{m-1})=0.
\end{equation*}
This problem is denoted as TCP$(\mathcal{B},p).$ For $m=2$ the tensor complementarity reduces to a linear complementarity problem. The linear complementarity problem is defined as follows. Given a matrix $B\in \mathbb{R}^{n \times n}$ and a vector $p\in \mathbb{R}^n,$ the linear complementarity problem, denoted by LCP$(B,p)$ is to find $v \in \mathbb{R}^n$ which satisfies the following conditions
\begin{equation*}\label{linear comp problem}
		p + B v \in \mathbb{R}^n_+, \;\;\;\; v \in \mathbb{R}^n_+, \;\;\;\; \mbox{and}\;\;\;\; v^{T}(p + Bv)=0.
	\end{equation*}

\NI The concept of complementarity takes into account a wide range of optimization problem. Linear programming, linear fractional programming, convex quadratic programming, and the bimatrix game problem are among the issues that can be classified as linear complementarity problems. It is widely used in operations research, control theory, mathematical economics and engineering. It is also well-recognized in the literature on mathematical programming. For recent works on this problem and applications see \cite{mohan2001more}, \cite{mohan2001classes}, \cite{neogy2006some}, \cite{neogy2005almost}, \cite{mohan2004note}, \cite{das2018invex}, \cite{dutta2023some}, \cite{jana2021iterative}, \cite{jana2018semimonotone}, \cite{neogy2008mixture} and \cite{neogy2016optimization} references cited therein.

\NI The well-known linear complementarity problem served as the initial inspiration for the notion of PPT. The PPT fundamentally involves transforming the matrix of a linear system so that unknowns can be exchanged for the relevant entries. For details see \cite{das2017finiteness}, \cite{mondal2016discounted}, \cite{neogy2012generalized}, \cite{das2016properties} and \cite{neogy2005principal}. Numerous matrix classes and their subclasses have undergone extensive research as a result of computing, complexity theory, and the theoretical foundations of linear complementarity problems. For recent work on this problem and applications see \cite{jana2019hidden}, \cite{das2016generalized}, \cite{dutta2022on}, \cite{neogy2011singular}, \cite{neogy2009modeling}, \cite{das2018some}, \cite{jana2018processability}. For game problem and multivariate analysis see \cite{mondal2016discounted}, \cite{jana2021more}, \cite{jana2018processability}, \cite{neogy2013weak}, \cite{neogy2008mathematical}, \cite{neogy2005linear} and references cited therein.

\NI Many mathematical concepts are extended in the framework of set-valued mapping. The set-valued mapping becomes famous after Kakutani fixed point theorem, a generalization of Brouwer fixed point theorem. A function of  a set-valued mapping maps an individual element of a set to a subset of a set. We consider set-valued mappings in case of modeling uncertainties, disturbances or errors for which stability is an issue to study ill-posed or inverse problems. Here solution set has some topological properties such as semi-continuity, well-posedness and other related properties. In this case lack of bijection and continuity are the issues that does not allow to characterize many important properties of the function along with its domain or codomain. A significant component of sensitivity analysis \cite{facchinei2007finite} of complementarity problem is the set-valued complementarity problem. Set-valued nonlinear complementarity problem provides a unified framework for nonlinear complementarity problem, exteded complementarity problem, implicit complementarity problem, quasi-variational inequality, mixed nonlinear complementarity problem and minimax programming problem. For details on set-valued complementarity problem see \cite{zhou2013set}. Motivated by this study, here we introduce the set-valued tensor complementarity problem as a subclass of set-valued complementarity problem. It is difficult to study set-valued complementarity problem but by imposing the tensor structure we can have some existence and uniqueness theorems.
	
The paper is organized as follows. Section 2 presents some basic notations and results which will be needed in the subsequent sections. In section 3 we introduce the set-valued tensor complementarity problem and prove the necessary condition for the feasible solution set. We establish a connection between two sets to study $S$-tensor. We consider the merit function for the set-valued tensor complementarity problem and establish the necessary and sufficient conditions for the level boundedness in connection with limit $R_0$-property of set of tensors.\\

\section{Preliminaries}
The basic ideas and the notation used throughout the text are outlined in this section. All the vectors, matrices and tensors considered here are with real entries. For any positive integer $n,$ the set $\{ 1,...,n \}$ is denoted by $[n]$ . Denote $N_\infty =\cup_{n=1}^{\infty} \{ \{ n, n+1,... \}\}.$ Let $\mathbb{R}^n$ denote the $n$-dimensional Euclidean space and $\mathbb{R}^n_+ =\{ v\in \mathbb{R}^n : v\geq 0 \}.$ Any vector $v\in \mathbb{R}^n$ is a column vector unless specified otherwise. Let $e$ be the vector with all component being $1.$ The Euclidean norm of a vector $v$ is defined as $\|v\|_2 = \sqrt{|v_1^2| + \cdots |v_n^2|}.$ An $m$th order $n$ dimensional real tensor $\mathcal{B}= (b_{i_1 ... i_m}) $ is a multidimensional array of entries $b_{i_1 ... i_m} \in \mathbb{R}$ where $i_j \in [n]$ with $j\in [m]$. The set of all $m$th order $n$ dimensional real tensors are denoted by $T_{m,n}.$ Shao \cite{shao2013general} introduced a product of tensors. Let $\mathcal{A}$ with order $q \geq 2$ and $\mathcal{B}$ with order $k \geq 1$ be two $n$-dimensional tensors. The product of $\mathcal{A}$ and $\mathcal{B}$ is a tensor $\mathcal{C}$ of order $(m-1)(k-1) + 1$ and dimension $n$ with entries $c_{j \alpha_1 \cdots \alpha_{m-1} } =\sum_{j_2, \cdots ,j_m \in [n]} a_{j j_2 \cdots j_m} b_{j_2 \alpha_1} \cdots b_{j_m \alpha_{m-1}}$ where $j \in [n] $, $\alpha_1, \cdots, \alpha_{m-1} \in [n]^{k-1}.$

\NI Then for a tensor $\mathcal{B}\in T_{m,n}$ and $v\in \mathbb{R}^n,\; \mathcal{B}v^{m-1}\in \mathbb{R}^n $ is a vector defined by
	\[ (\mathcal{B} v^{m-1})_i = \sum_{i_2, ...,i_m =1}^{n} b_{i i_2 ...i_m} v_{i_2} \cdots v_{i_m} , \;\forall \; i \in [n], \]
	and $\mathcal{B}v^m\in \mathbb{R} $ is a scalar defined by
 \begin{equation*}
     v^T \mathcal{B}v^{m-1} = \mathcal{B}v^m = \sum_{i_1,...,i_m =1}^{n} b_{i_1  ...i_m} v_{i_1}  \cdots v_{i_m}.
 \end{equation*}	

\noindent Given a vector $p \in \mathbb{R}^n$ and a tensor $\mathcal{B} \in T_{m,n}$ the set of feasible solution of TCP$(\mathcal{B},p)$ is defined as FEA$(\mathcal{B},p)= \{v\in \mathbb{R}^n_+ : p + \mathcal{B}v^{m-1} \geq 0\}$ and the solution set of TCP$(\mathcal{B},p)$ as SOL$(\mathcal{B},p)= \{v\in$ FEA$(\mathcal{B},p) :  v^{T}(p + \mathcal{B}v^{m-1})= 0\}.$ 

\NI We consider some definitions and results which are required for the next sections.

\begin{defn}\cite{shao2016some}
The $i$th row subtensor of $\mathcal{B}= (b_{i_1 ... i_m}) \in T_{m,n}$ is denoted by $R_i(\mathcal{B})$ and its entries are given as $(R_i(\mathcal{B}))_{i_2 ... i_m}=(b_{i i_2... i_m})$ where $i_j\in [n]$ and $2\leq j\leq m.$
\end{defn}

\begin{defn}\cite{wets1998variational}
A function $f: \mathbb{R}^n \to \mathbb{R}^n$ is said to be level bounded if for every $\alpha \geq 0$ the level set $\{v\in \mathbb{R}^n: f(v) \leq \alpha \}$ is bounded.
\end{defn}

\begin{defn}\cite{song2015properties}
A tensor $\mathcal{B}\in T_{m,n} $ is said to be a $Q$-tensor if for every $p\in \mathbb{R}^n$, TCP$(\mathcal{B},p)$ has a solution.
\end{defn}

\begin{defn}\cite{song2015properties}
A tensor $\mathcal{B}\in T_{m,n} $ is said to be a $S$-tensor if the system $\mathcal{B}v^{m-1} >0, \; v > 0 $ has a solution.
\end{defn}

\begin{defn}\cite{deb2023more}, \cite{song2017properties}
A tensor $\mathcal{B}\in T_{m,n} $ is (strictly) semipositive tensor if for each $v\in \mathbb{R}^n_+ \backslash \{0\}$, $\exists \; k\in [n]$ such that $v_k>0$ and $(\mathcal{B}v^{m-1})_k\; \geq 0\; (>0)$.
\end{defn}

\begin{defn}\cite{song2016properties}, \cite{song2015properties}
A tensor $\mathcal{B}\in T_{m,n} $ is said to be a $R_0$-tensor if the TCP$(\mathcal{B},0)$ has unique zero solution. 
\end{defn}

\begin{theorem}\cite{song2016properties}\label{boundedness theorem for R_0-tensor}
    If $\mathcal{B}\in T_{m,n}$ is a an $R_0$-tensor then for $p \in \mathbb{R}^n,$ the solution set of the TCP$(\mathcal{B},p)$ is bounded.
\end{theorem}

\begin{defn}
     For a mapping $F: \mathbb{R}^n \to \mathbb{R}^m,$ define
\begin{equation*}
    \liminf_{v\to \bar{v}} F(v)= \left( \begin{array}{c}
        \liminf_{v\to \bar{v}} F_1(v) \\
        \vdots \\
        \liminf_{v\to \bar{v}} F_m(v)
    \end{array} \right)
\end{equation*}
\end{defn}

\begin{defn}\cite{aubin2009set}, \cite{wets1998variational}
    Given a set-valued mapping $H:\mathbb{R}^n \rightrightarrows \mathbb{R}^m,$ define $\limsup_{v\to \bar{v}}$ $H(v) = \{ u : \;\exists \; v^n \to \Bar{v}, \; \exists \; u^n \to u \mbox{ with } u^n \in H(v^n) \},$ $\liminf_{v\to \bar{v}} H(v) = \{ u : \;\forall \; v^n \to \Bar{v}, \; \exists \; u^n \to u \mbox{ with } u^n \in H(v^n) \},$ $\limsup_{v\to \infty} H(v) = \{ u : \;\exists \; v^n \to \infty, \; \exists \; u^n \to u \mbox{ with } u^n \in H(v^n) \}$ and $ \liminf_{v\to \infty} H(v) = \{ u : \;\forall \; v^n \to \infty, \; \exists \; u^n \to u \mbox{ with } u^n \in H(v^n) \}.$
\end{defn}

\begin{defn}\cite{aubin2009set}
    A function $H:\mathbb{R}^n \rightrightarrows \mathbb{R}^m,$ is outer semicontinuous at $\Bar{v},$ if, $\limsup_{v\to \bar{v}}$ $H(v) \subset H(\Bar{v})$ and inner semicontinous at $\Bar{v},$ if $\liminf_{v\to \bar{v}} H(v) \supset H(\Bar{v}).$ The function $H$ is continous at $\Bar{v}$ if it is both outer semicontinous and inner semicontinous at $\Bar{v}.$
\end{defn}

\begin{defn}\cite{aubin2009set}
    Let $V$ and $W$ be two metric spaces. A set-valued map $H$ from $V$ to $W$ is characterised by its graph which is denoted by Graph$(H)$ and is defined by
    \begin{equation*}
        \mbox{Graph}(H)=\{ (v,w)\in V\times W : \; w \in H(v) \}.
    \end{equation*}
\end{defn}

\begin{remk}
There are two different ways to define the inverse image of a subset $M$ by a set-valued map $H$:\\
\NI (a) $H^{-1}(M)=\{ v: H(v)\cap M \neq \emptyset \}$.\\
\NI (b) $H^{+1}(M)=\{ v: H(v)\subset M \}$.\\
Here $H^{-1}(M)$ is said to be the inverse image of $M$ by $H$ and $H^{+1}(M)$ is said to be the core of $M$ by $H.$
\end{remk}

\begin{theorem}\cite{aubin2009set}\label{theorem 1 for semicotinuity}
    A set-valued mapping $H:V\rightrightarrows W$ is upper semicontinuous at $v$ if the core of any neighbourhood of $H(v)$ is a neighbourhood of $v$ and a set-valued map is lower semicontinuous at $v$ if the inverse image of any open subset intersecting $H(v)$ is a neighbourhood of $v.$  
\end{theorem}

\begin{theorem}\cite{aubin2009set}\label{theorem 2 for semicotinuity}
Consider a metric space $V,$ two normed spaces $W$ and $Z,$ two set-valued maps $G:V\rightrightarrows W$ and $H:V\rightrightarrows Z$ respectively, and a (single-valued) map $h:V\times Z\rightarrow W$ satisfying the following assumptions:\\
\NI (a) $G$ and $H$ are lower semicontinuous with convex values\\   
\NI (a) $h$ is continuous\\
\NI (a) $u \mapsto h(v,u)$ is affine for all $v\in V$\\
\NI We posit the following condition:\\
\NI $\forall \; v\in V, \exists\; \alpha, \; \delta,\; c,\; \beta >0$ such that $\forall \; v^{\prime} \in B(v,\delta)$ we have $\alpha B_W \subset h(v^{\prime}, H(v^{\prime})\cap \beta B_Z)- G(v^{\prime}).$\\
Then the set-valued map $R:V \rightrightarrows Z$ defined by $R(v)= \{ u\in H(v): h(v,u) \in G(v)\}$ is lower semicontinuous with nonempty convex values.
\end{theorem}

In this paper we always assume that the set-valued mapping $\Omega$ is closed valued, i.e. $\Omega(v)$ is closed $\forall \; v\in \mathbb{R}^n.$

\section{Main results}

We begin by introducing the definition of set-valued tensor complementarity problem (SVTCP$(\mathcal{B}, p, \Omega)$).\\

The set-valued tensor complentarity problem (SVTCP$(\mathcal{B}, p, \Omega)$) is defined as follows:
\NI Find a vector $v\in \mathbb{R}^n$ such that
\begin{equation}\label{SICP}
    v \geq 0,\;\; H(v,\omega) \geq 0, \;\; v^T H(v,\omega) = 0,\;\; \omega \in \Omega(v)
\end{equation}
where $H:\mathbb{R}^n \times \mathbb{R}^m \to \mathbb{R}^n$ and $\Omega :\mathbb{R}^n \rightrightarrows \mathbb{R}^m$ are set-valued mapping, $F= \mathcal{B}(\omega)v^{m-1} + p(\omega)$ and $\Omega$ is a set in $\mathbb{R}^k.$ The solution set of SVTCP$(\mathcal{B}(\omega),p(\omega), \Omega)$ is denoted by SOL$(\mathcal{B}(\omega),p(\omega),\Omega)= \{v: v\geq0,\; \mathcal{B}(\omega)v^{m-1}+ p(\omega) \geq 0,\;  v^{T}(\mathcal{B}(\omega)v^{m-1}+ p(\omega))= 0, \; \forall\; \omega \in \Omega\}.$\\

\NI In the theory of tensor complementarity problem an $S$-tensor gives condition for the feasibility of the solution set of TCP. Consider the sets $A=\{v: v> 0, \; \mathcal{B}v^{m-1} > 0\}$ and $A^{\prime}=\{v: v\neq 0,\; v \geq 0,\;  \mathcal{B}v^{m-1} > 0\}.$ In tensor complementarity theory for an $S$-tensor $A$ is nonempty. Clearly $A\neq \emptyset \implies A^{\prime} \neq \emptyset.$ For an $S$-tensor $A^{\prime}\neq \emptyset \implies A \neq \emptyset.$ However this kind of implication fails to hold in case of set-valued complementarity problem. i.e., for the sets $C=\{v: v> 0,\;  \mathcal{B}(\omega)v^{m-1} > 0, \; \mbox{ for some } \omega \in \Omega (v) \}$ and $C^{\prime} = \{ v: v\neq0,\; v \geq 0, \;  \mathcal{B}(\omega)v^{m-1} > 0, \; \mbox{ for some } \omega \in \Omega (v)\},$ $C^{\prime}\neq \emptyset$ does not imply $C\neq \emptyset.$ Consider the following example which illustrates this phenomenon.

\begin{examp}
Let $\mathcal{B}(\omega)$ be of order $3$ and of dimension $2.$ Let the elements of $\mathcal{B}(\omega)$ be given as $b_{111}(\omega) = b_{222}(\omega) = \omega,$ and other elements are zero. Let $\Omega (v)= \left\{ \begin{array}{cc}
        \{0,1\}, & v=(1,0)^T \in \mathbb{R}^2 ;\\
        \{0\}, & \mbox{ otherwise. }
    \end{array} \right.$
If $\mathcal{B}(\omega) v^{m-1} >0,$ then $\omega =1,$ and such case holds only when $v=(1,0)^T.$ Therefore, $v\in C^{\prime}.$ However $v\notin C$ and $C=\emptyset.$
\end{examp}

Here we present a condition under which for an $S$-tensor we claim $C^{\prime}\neq \emptyset \implies C\neq \emptyset.$

\begin{theorem}
    If $\Omega(v)$ is inner semicontinuous and $\mathcal{B}(\omega)$ is continuous then $C^{\prime}\neq \emptyset \implies C\neq \emptyset.$
\end{theorem}
\begin{proof}
Let $F(v)= \max_{\omega \in \Omega(v)} \mathcal{B}(\omega)v^{m-1}$ and $R_i (\mathcal{B})(\omega)$ denote the $i$th row subtensor of $\mathcal{B}(\omega).$ Hence $F_i(v)= \max_{\omega \in \Omega(v)} R_i (\mathcal{B})(\omega)v^{m-1}.$ For a given $v_0 \in \mathbb{R}^n$ then for $\varepsilon > 0$ $\exists \; \omega_0 \in \Omega(v_0)$ such that $R_i (\mathcal{B})(\omega_0)v_0^{m-1} > F_i(v_0) - \varepsilon .$ Since $\Omega(v)$ is inner semicontinuous by Theorem \ref{theorem 1 for semicotinuity} and Theorem \ref{theorem 2 for semicotinuity}, for any sequence $\{v_n\}$ such that $v_n \to v_0, \; \exists \; \{\omega_n\} \subset \Omega(v_n),$ $\omega_n \to \omega_0.$ This implies $F_i(v_n)= \max_{\omega \in \Omega(v_n)} R_i (\mathcal{B})(\omega)v_n^{m-1} \geq R_i (\mathcal{B})(\omega_n)v_n^{m-1}.$ This implies,
    \begin{equation*}
        \liminf_{n \to \infty} F_i(v_n) \geq \lim_{n\to \infty} R_i (\mathcal{B})(\omega_n)v_n^{m-1} = R_i (\mathcal{B})(\omega_0)v_0^{m-1} > F_i(v_0) - \varepsilon.
    \end{equation*}
    Here the equality follows from the continuity of $R_i (\mathcal{B})(\omega)$ which is ensured by the continuity of $\mathcal{B}(\omega).$ Since $\varepsilon >0$ and $\{v_n\}$ is an arbitrary sequence converging to $v_0,$ we obtain
\begin{equation*}
    \liminf_{v\to v_0} F_i(v) \geq F_i(v_0).
\end{equation*}
This implies the lower semicontinuity of $F_i$. Again,

\begin{equation*}
    \liminf_{v\to v_0} F(v)
    = \left( \begin{array}{c}
        \liminf_{v\to v_0} F_1(v) \\
        \vdots \\
        \liminf_{v\to v_0} F_n(v)
             \end{array}  \right)
    \geq \left( \begin{array}{c}
        F_1(v_0) \\
        \vdots \\
        F_n(v_0)
         \end{array}  \right)
    = F(v_0).     
\end{equation*}
Therefore,
\begin{equation*}
    \liminf_{v\to v_0} \max_{\omega \in \Omega(v)} \mathcal{B}(\omega) v^{m-1} \geq \max_{\omega \in \Omega(v_0)} \mathcal{B}(\omega) v_0^{m-1}.
\end{equation*}
Let $C^{\prime}\neq \emptyset.$ If $\Bar{v} \in C^{\prime},$ then
\begin{equation*}
    \Bar{v} \geq 0, \;\;\;\; \mathcal{B}(\Bar{\omega}) \Bar{v}^{m-1} >0, \;\;\;\; \mbox{ for some } \Bar{\omega} \in \Omega(\Bar{v})
\end{equation*}
which is equivalent to
\begin{equation*}
    \Bar{v} \geq 0, \;\;\;\; F(\Bar{v}) >0.
\end{equation*}
Again, $\liminf_{t\to 0^+} F(\Bar{v} + t e) \geq F(\Bar{v}) >0,$ and $\Bar{v} + t e>0$ for $t>0.$ For sufficiently small $t>0$ we have $(\Bar{v} + t e) \in C.$ Hence $C\neq \emptyset.$ 
\end{proof}

Here we propose a necessary condition for the feasibility of SVTCP$(\mathcal{B}(\omega), p(\omega), \Omega).$

\begin{theorem}\label{theorem with feasibility}
Consider the SVTCP$(\mathcal{B}(\omega), p(\omega), \Omega).$ If $\exists \; v\in \mathbb{R}^n$ such that
\begin{equation*}\label{condition for feasibility}
    v\geq 0, \;\;\;\; \mathcal{B}(\omega) v^{m-1} >0 \;\;\;\; \mbox{ for some } \omega \in \cap_{\Bar{N} \in N_{\infty}} \cup_{n\in \Bar{N}} \Omega(nv),
\end{equation*}
and $p(\omega): \mathbb{R}^m \to \mathbb{R}^n$ are bounded below then SVTCP$(\mathcal{B}(\omega), p(\omega), \Omega)$ is feasible.
\end{theorem}
\begin{proof}
Let the mapping $p(\omega)$ be bounded below. Therefore $\exists \; \gamma \in \mathbb{R},$ such that $p(\omega) \geq \gamma e.$ Suppose $v_0$ and $\omega_0$ satisfy the given assumption. Hence we obtain
\begin{equation*}
    v_0 \geq 0, \;\;\;\; \mathcal{B}(\omega_0) v_0^{m-1} >0, \;\;\;\; \omega_0 \in \cap_{\Bar{N} \in N_{\infty}} \cup_{n\in \Bar{N}} \Omega(n v_0).
\end{equation*}
Then, for any $\Bar{N} \in N_{\infty},$ we have $\omega_0 \in \cup_{n\in \Bar{N}} \Omega (n v_0).$ In particular, we observe the following:

\NI (a) If $\Bar{N}= \{ 1,2,... \},$ then $\exists \; n_1,$ such that $\omega_0 \in \Omega(n_1 v_0);$

\NI (b) If  $\Bar{N}= \{ n_1 +1,... \}$ then $\exists \; n_2$ with $n_2> n_1,$ such that $\omega_0 \in \Omega(n_2 v_0).$

\NI By this process we find a sequence $\{n_l\}$ such that $\omega_0 \in \Omega(n_l v_0)$ and $n_l \to \infty.$ Since $\mathcal{B}(\omega_0) v_0^{m-1} >0,$ $\exists \; \delta > 0,$ such that $\mathcal{B}(\omega_0) v_0^{m-1} > \delta e.$ Taking $l$ large enough to satisfy $n_l > \max\{ -\frac{\gamma}{\delta}, 0 \}$ we write $\delta n_l e > -\gamma e \geq -p(\omega).$ This implies that
\begin{equation*}
    \mathcal{B}(\omega_0) (n_l v_0)^{m-1} > \delta n_l^{m-1} e \geq -p(\omega).
\end{equation*}
Hence,
\begin{equation*}
    n_l v_0 \geq 0, \;\; \mathcal{B}(\omega_0) (n_l v_0)^{m-1} + p(\omega) > 0, \;\; \omega_0 \in \Omega(n_l v_0).
\end{equation*}
This implies $n_l v_0$ is a feasible point of SVTCP$(\mathcal{B}(\omega), p(\omega), \Omega).$
\end{proof}

\begin{defn}
The set of tensors $\{ \mathcal{B}(\omega) : \omega \in \Omega(v) \}$ is said
to be\\
(a) strongly semi-positive if for any nonzero $v\geq 0,$ $\exists \; v_k> 0$ such that
\begin{equation*}\label{strongly semipositive}
    (\mathcal{B}(\omega) v^{m-1})_k \geq 0,\;\; \forall \; \omega \in \Omega(v).
\end{equation*}
(b) weakly semi-positive if for any nonzero such that $v\geq 0,$ $\exists \; v_k> 0$ such that
\begin{equation*}\label{weakly semipositive}
    (\mathcal{B}(\omega) v^{m-1})_k \geq 0,\;\; \mbox{for some} \; \omega \in \Omega(v).
\end{equation*}
\end{defn}

\begin{theorem}
For the SVTCP$(\mathcal{B}(\omega),p(\omega), \Omega),$ the following statements hold:\\
(a) If the set of tensors $\{ \mathcal{B}(\omega) : \omega \in \Omega(v) \}$ is strongly semi-positive, then for any positive mapping $p,$ i.e, $p(\omega) >0$ for all $\omega,$ SVTCP$(\mathcal{B}(\omega),p(\omega), \Omega),$ has zero as its unique solution.\\
(b) If SVTCP$(\mathcal{B}(\omega),p(\omega), \Omega)$ with $p(\omega) >0$  has zero as its unique solution, then the set of tensors $\{ \mathcal{B}(\omega) : \omega \in \Omega(v) \}$ is weakly semi-positive.
\end{theorem}
\begin{proof}
(a) We prove this part by showing that if $\exists$ any non-zero solution then the set of tensors $\{ \mathcal{B}(\omega) : \omega \in \Omega(v) \}$ is not strongly semi-positive. Note that, for any positive mapping $p, \; v = 0$ is a solution of SVTCP$(\mathcal{B}(\omega),p(\omega), \Omega).$ We assume that there is another nonzero solution $\Tilde{v},$ i.e, $\exists\; \Tilde{ \omega} \in \Omega(\Tilde{v}),$ such that
\begin{equation}\label{for proof of strong semipositive} 
    \Tilde{v} \geq 0, \;\; \mathcal{B}(\Tilde{\omega})\Tilde{v}^{m-1} + p(\Tilde{\omega}) \geq 0,\;\;   \Tilde{v}^T(\mathcal{B}(\Tilde{\omega})\Tilde{v}^{m-1} + p(\Tilde{\omega})) = 0
\end{equation}
We want to prove that the set of tensors $\{ \mathcal{B}(\omega) : \omega \in \Omega(v) \}$ is not strongly semi-positive. Let the set of tensors $\{ \mathcal{B}(\omega) : \omega \in \Omega(v) \}$ be strongly semi-positive. Then there exists $k\in [n]$ such that
$\Tilde{v}_k >0 $ and $(\mathcal{B}(\Tilde{\omega})\Tilde{v}^{m-1})_k \geq 0,$ and hence $(\mathcal{B}(\Tilde{\omega})\Tilde{v}^{m-1} + p(\Tilde{\omega}))_k > 0,$ which contradicts condition (\ref{for proof of strong semipositive}). Hence the set of tensors $\{ \mathcal{B}(\omega) : \omega \in \Omega(v) \}$ is not strongly semi-positive.\\

\NI (b) We prove this part by showing that if the set of tensors $\{ \mathcal{B}(\omega) : \omega \in \Omega(v) \}$ is not weakly semi-positive then the SVTCP$(\mathcal{B}(\omega),p(\omega), \Omega)$ with $p(\omega) >0$ does not have zero as its unique solution. We assume that $\{ \mathcal{B}(\omega) : \omega \in \Omega(v) \}$ is not weakly semi-positive. Then, there exists a nonzero $\Tilde{v} \geq 0,$ for all $k\in I^+(\Tilde{v}) = \{ i: \Tilde{v}_i > 0, (\mathcal{B}(\omega)\Tilde{v}^{m-1})_k < 0 \}$ for all $\omega \in \Omega (\Tilde{v}).$ Choose $\tilde{\omega} \in \Omega (\Tilde{v}).$ Let $p(\omega)= 1$ for all $\omega \neq \tilde{\omega}$ and 
\begin{equation*}
(p(\tilde{\omega}))_k =\left\{ \begin{array}{ll}
    -(\mathcal{B}(\tilde{\omega})\Tilde{v}^{m-1})_k, & k\in I^+(\Tilde{v}); \\
    \max\{ (\mathcal{B}(\tilde{\omega})\Tilde{v}^{m-1})_k, 0 \} +1, &  \mbox{ otherwise.}
\end{array} \right.
\end{equation*}
Therefore $p(\omega) >0$ for all $\omega.$ According to the above construction, we have
\begin{equation*}
    \Tilde{v} \geq 0, \;\; \mathcal{B}(\Tilde{\omega})\Tilde{v}^{m-1} + p(\Tilde{\omega}) \geq 0,\;\;   \Tilde{v}^T(\mathcal{B}(\Tilde{\omega})\Tilde{v}^{m-1} + p(\Tilde{\omega})) = 0, \mbox{ with } \Tilde{\omega} \in \Omega (\Tilde{v}).
\end{equation*}
This implies the nonzero vector $\Tilde{v}$ is a solution of SVTCP$(\mathcal{B}(\omega),p(\omega), \Omega).$
\end{proof}

\begin{theorem}
    If the set of tensor $\{\mathcal{B}(\omega): \omega\in \Omega(v)\}$ is strongly semi-positive and SOL$\{\mathcal{B}(\omega), 0\} = \{0\}, \forall \; \omega \in \Omega,$ then SOL$(\mathcal{B}(\omega), p(\omega), \Omega)$ is nonempty.
\end{theorem}
\begin{proof}
    Since set of tensor $\{\mathcal{B}(\omega): \omega\in \Omega(v)\}$ is strongly semipositive, $\forall\; p(\omega)>0$ SVTCP$(\mathcal{B}(\omega),p(\omega), \Omega)$ has zero as its unique solution. Also SOL$\{\mathcal{B}(\omega), 0, \Omega\} = \{0\}.$ Thus for every $\omega \in \Omega, \; \exists \; d>0$ such that for TCP$(\mathcal{B}(\omega), 0)$ and TCP$(\mathcal{B}(\omega), d)$ zero is the only solution. Therefore by Corollary 3.2, Item (i) of \cite{gowda2015z} we conclude that $\mathcal{B}(\omega)$ is a $Q$-tensor for each $\omega \in \Omega (v).$ This implies SOL$(\mathcal{B}(\omega), p(\omega),\Omega)$ is nonempty. 
\end{proof}

\begin{theorem}
    For the SVTCP$(\mathcal{B}(\omega), p(\omega), \Omega)$ if $\exists \; \omega_0 \in \Omega$ such that $\mathcal{B}(\omega_0)$ is a $P$-tensor with $\Omega^{-1} (\omega_0) \subseteq$SOL$(\mathcal{B}(\omega_0), p(\omega_0)),$ then SOL$(\mathcal{B}(\omega), p(\omega), \Omega)$ is nonempty.
\end{theorem}
\begin{proof}
    Consider the SVTCP$(\mathcal{B}(\omega), p(\omega), \Omega).$ Suppose $\omega_0\in \Omega$ is such that $\mathcal{B}(\omega_0)$ is a $P$-tensor then SOL$(\mathcal{B}(\omega_0), p(\omega_0))$ is nonempty and compact. Also suppose that $\Omega^{-1}(\omega_0) \subseteq $ SOL$(\mathcal{B}(\omega_0), p(\omega_0)).$ Now from the Theorem 17 of \cite{zhou2013set}, we know
    \begin{equation*}
        \mbox{SOL}(\mathcal{B}(\omega), p(\omega), \Omega) = \bigcup_{\omega \in \mathbb{R}^m} [\mbox{SOL}(\mathcal{B}(\omega), p(\omega)) \cap  \Omega^{-1}(\omega)].
    \end{equation*}
    Since $\Omega^{-1}(\omega_0) \subseteq $ SOL$(\mathcal{B}(\omega_0), p(\omega_0))$ and SOL$(\mathcal{B}(\omega_0), p(\omega_0))$ is nonempty, we have 
    \[\bigcup_{\omega \in \mathbb{R}^m} [\mbox{SOL}(\mathcal{B}(\omega), p(\omega)) \cap  \Omega^{-1}(\omega)] \neq \phi.\]
\end{proof}

\NI Before starting the next result, set-valued tensor complementarity problem is to transfer to a system of equations or an unconstrained optimization via merit function. We characterise the solution set of set-valued tensor complementarity problem by establishing the necessary and sufficient conditions for the level boundedness of the merit function.

\NI A function $f: \mathbb{R}^n \to \mathbb{R}$ is said to be a merit function or residual function for a complementarity problem if\\
\NI (a) $f(v) \geq 0\; \forall \; v\in \mathbb{R}^n.$\\
\NI (b) $f(v) = 0$ if and only if $x$ is a solution of the complementarity problem.

\NI For SVTCP$(\mathcal{B}(\omega), p(\omega), \Omega)$ we define the merit function as 
$$r(v) = \min_{\omega \in \Omega(v)} \| \min \{v, \mathcal{B}(\omega)v^{m-1} + p(\omega)\}\|.$$

\NI We define the limit-$R_0$ property of a set of tensors. This definition extends the definition of $R_0$-tensor in the case of set-valued tensor complementarity problem.

\begin{defn}
The set of tensors $\{ \mathcal{B}(\omega): \omega \in \Omega(v) \}$ corresponding to the SVTCP $(\mathcal{B}(\omega),p(\omega), \Omega)$ is said to have the limit-$R_0$ property if
\begin{equation*}
    v\geq 0,\; \; \mathcal{B}(\omega)v^{m-1} \geq 0,\;\; \mathcal{B}(\omega)v^m =0 \mbox{ for some } \omega \in \limsup_{x \rightarrow \infty} \Omega(v) \implies v=0.
\end{equation*}
\end{defn}

\NI We show that the limit-$R_0$ property provides a necessary condition for a merit function to be level bounded.

\begin{theorem}
Consider the SVTCP$(\mathcal{B}(\omega),p(\omega), \Omega).$ Suppose $\exists$ a bounded set $\Omega,$ such that $\Omega(v) \subset \Omega \; \; \forall\; v \in \mathbb{R}^n$ and $\mathcal{B}(\omega)$ and $p(\omega)$ are continuous on $\Omega.$ If the set of tensors $\{ \mathcal{B}(\omega): \omega \in \Omega(v) \}$ has the limit-$R_0$ property, then the merit function, $r(v) = \min_{\omega \in \Omega(v)} \| \min \{v, \mathcal{B}(\omega)v^{m-1} + p(\omega)\}\|$ is level bounded.
\end{theorem}
\begin{proof}
We establish the result by showing that if $r(v)$ is not level bounded for the SVTCP$(\mathcal{B}(\omega),p(\omega), \Omega)$ where $\exists$ a bounded set $\Omega,$ such that $\Omega(v) \subset \Omega \; \; \forall\; v\in \mathbb{R}^n$ and $\mathcal{B}(\omega)$ and $p(\omega)$ are continuous on $\Omega$ then the set of tensors $\{ \mathcal{B}(\omega): \omega \in \Omega(v) \}$ does not have the limit-$R_0$ property. By this approach we first assume that $r(v)$ is not level bounded. Then $\exists$ a sequence $\{v_n\},$ such that $\|v_n\| \to \infty$ as $n \to \infty$ and $\{r(v_n)\}$ is bounded. Then,
\begin{align}\label{equation for weak error bound  1st part}
    \frac{r(v_n)}{\|v_n\|^{m-1}} & = \min_{\omega \in \Omega(v_n)} \left\| \min \left\{ \frac{v_n}{\|v_n\|^{m-1}}, \frac{\mathcal{B}(\omega)v^{m-1} + p(\omega)}{\|v_n\|^{m-1}} \right\} \right\| \notag \\ 
    & = \left\| \min \left\{ \frac{v_n}{\|v_n\|^{m-1}}, \frac{\mathcal{B}(\omega_n)v^{m-1} + p(\omega_n)}{\|v_n\|^{m-1}} \right\} \right\|,
\end{align}
the minimizer is attained at $\omega_n \in \Omega (v_n),$ whose  existence is ensured by the compactness of $\Omega(v_n),$ since $\Omega(v)$ is closed and $\Omega$ is bounded. Taking a subsequence if necessary, we can assume that $\frac{v_n}{\|v_n\|^{m-1}}$ converges to $v_0$ and $\{\omega_n\}$ converges to $\omega_0.$ We have,
\begin{equation*}
    \omega_0 \in \limsup_{n\to \infty} \Omega(v_n) \subseteq \limsup_{ \|v\| \to \infty} \Omega(v).
\end{equation*}
Since $p(\omega)$ is continuous and $\omega_n \in \Omega$ is bounded we have $\lim_{n\to \infty} \frac{p(\omega_n)}{\|v_n\|^{m-1}}=0.$

Taking limit on both sides of the equation (\ref{equation for weak error bound  1st part}) as $n\to \infty$ we obtain,
\begin{equation*}
  \left\| \min\left\{v_0, \mathcal{B}(\omega_0)v_0^{m-1} \right\} \right\| = 0 .
\end{equation*}
Since $v_0$ is a nonzero vector, this implies that the set of tensors $\{ \mathcal{B}(\omega): \omega \in \Omega(v) \}$ does not have limit-$R_0$ property.
\end{proof}

\begin{theorem}
For SVTCP$(\mathcal{B}(\omega),p(\omega), \Omega),$ suppose $\exists$ a compact set $\Omega,$ such that $\Omega(v) \subset \Omega, \; \forall$ $v\in \mathbb{R}^n,$ where $p(\omega)$ and $ \mathcal{B}(\omega)$ are continuous on $\Omega.$ If the merit function $r(v) = \min_{\omega \in \Omega(v)} \| \min(v, \mathcal{B}(\omega)v^{m-1} + p(\omega))\|$ is level bounded, then
\begin{equation*}\label{last theorem}
    v\geq 0,\; \; \mathcal{B}(\omega)v^{m-1} \geq 0,\;\; \mathcal{B}(\omega)v^m =0 \mbox{    for some } \omega \in \cap_{\Bar{N} \in N_{\infty}} \cup_{n\in \Bar{N}} \Omega(nv) \implies v= 0.
\end{equation*}
\end{theorem}
\begin{proof}
We prove this by showing that if there exists $v_0 \neq 0$ and $\omega_0 \in \cap_{\Bar{N} \in N_{\infty}} \cup_{n\in \Bar{N}} \Omega(nv_0)$ for the SVTCP$(\mathcal{B}(\omega),p(\omega), \Omega)$ with $p(\omega)$ and $ \mathcal{B}(\omega),$ continuous on the compact set $\Omega$ and $\Omega(v) \subset \Omega, \; \forall$ $v\in \mathbb{R}^n$ such that  $v_0\geq 0,\; \; \mathcal{B}(\omega_0)v_0^{m-1} \geq 0,\;\; \mathcal{B}(\omega_0)v_0^m =0$
then the merit function $r(v) = \min_{\omega \in \Omega(v)} \| \min(v, \mathcal{B}(\omega)v^{m-1} + p(\omega))\|$ is not level bounded. Suppose $\exists \; v_0 \neq 0$ and $\omega_0 \in \cap_{\Bar{N} \in N_{\infty}} \cup_{n\in \Bar{N}} \Omega(nv_0)$ such that
\begin{equation}\label{last proof 1}
     v_0\geq 0,\; \; \mathcal{B}(\omega_0)v_0^{m-1} \geq 0,\;\; \mathcal{B}(\omega_0)v_0^m =0
\end{equation}
Then for any $\Bar{N} \in N_{\infty},$ we have $\omega_0 \in \cup_{n\in \Bar{N}} \Omega (n v_0).$ In particular, we observe the following:

\NI (a) If $\Bar{N}= \{ 1,2,... \},$ then $\exists \; n_1,$ such that $\omega_0 \in \Omega(n_1 v_0);$

\NI (b) If  $\Bar{N}= \{ n_1 +1,... \}$ then $\exists \; n_2$ with $n_2> n_1,$ such that $\omega_0 \in \Omega(n_2 v_0).$

\NI By this process we find a sequence $\{n_l\}$ such that $\omega_0 \in \Omega(n_l v_0)$ and $n_l \to \infty.$
 Therefore,
\begin{align}\label{last theorem 1} \notag
    r(n_l v_0) & = \min_{\omega \in \Omega(n_l v_0)} \| \min \{n_l v_0,\; n_l^{m-1}\mathcal{B}(\omega)v_0^{m-1} + p(\omega)\}\|\\ \notag
          & \leq \| \min \{n_l v_0,\; n_l^{m-1}\mathcal{B}(\omega_0)v_0^{m-1} + p(\omega_0)\}\| \\ 
          & \leq \sum_{i=1}^n \| \min \{n_l (v_0)_i,\; n_l^{m-1}(\mathcal{B}(\omega_0)v_0^{m-1})_i + p_i(\omega_0)\}\| 
\end{align}
Now we consider the following cases.

\NI Case-1: We consider the case when $(v_0)_i>0.$ Then using (\ref{last proof 1}) we conclude $(\mathcal{B}(\omega_0)$ $v_0^{m-1})_i = 0.$ Imposing the continuity of the function $p(\omega)$ and the compactness of the set $\Omega$ we conclude that $\max_{\omega \in \Omega}p(\omega)$ is bounded. This implies that $\exists$ sufficiently large $k$ such that $\max_{\omega \in \Omega}p_i(\omega) < n_l (v_0)_i.$ Hence,
\begin{equation}\label{last theorem 2}
    \| \min \{n_l (v_0)_i, n_l^{m-1} (\mathcal{B}(\omega_0)v_0^{m-1})_i + p_i(\omega_0)\}\| = \|p_i(\omega_0) \|
\end{equation}

\NI Case-2: We consider the case when $(v_0)_i=0.$ Here we consider the following two subcases.

\NI Subcase-1: When $ n_l^{m-1} (\mathcal{B}(\omega_0)v_0^{m-1})_i + p_i(\omega_0) \geq 0,$ we have
\begin{equation*}\label{subcase 1}
    \|\min\{n_l (v_0)_i, n_l^{m-1}(\mathcal{B}(\omega_0)v_0^{m-1})_i + p_i(\omega_0)\}\| = 0
\end{equation*}

\NI Subcase-2: When $n_l^{m-1} (\mathcal{B}(\omega_0)v_0^{m-1})_i + p_i(\omega_0) < 0,$ we obtain 
\begin{equation*}\label{subcase 2}
    \min \{n_l (v_0)_i,\;  n_l^{m-1}(\mathcal{B}(\omega_0)v_0^{m-1})_i + p_i(\omega)\}^2 \leq p_i(\omega_0)^2
\end{equation*}
using the fact $p_i(\omega_0) \leq n_l^{m-1}(\mathcal{B}(\omega_0)v_0^{m-1})_i + p_i(\omega_0) < 0.$ 
Thus combining we have
\begin{equation}\label{last theorem 3}
   \|\min\{n_l (v_0)_i,\; n_l^{m-1}(\mathcal{B}(\omega_0)v_0^{m-1})_i + p_i(\omega_0)\}\| \leq \| p_i(\omega_0)\|
\end{equation}
Putting the facts (\ref{last theorem 1}), (\ref{last theorem 2}) and (\ref{last theorem 3}) together we obtain,
\begin{equation*}
r(n_l v_0) \leq \sum_{i=1}^n \| p_i(\omega_0)\|.
\end{equation*}
This implies that $r(v)$ is not level bounded.
\end{proof}

\section{Conclusion}
In this paper, we introduce the set-valued tensor complementarity problem and provide a necessary condition for the feasibility of the solution set for SVTCP. For a set of strongly semi-positive tensors the SVTCP$(\mathcal{B}(\omega),p(\omega), \Omega)$ with $p(\omega) >0$ has zero as its unique solution. The weakly semi-positivity of the involved set of tensors provides a necessary condition for the uniqueness of zero solution of the SVTCP$(\mathcal{B}(\omega),p(\omega), \Omega)$ where $p(\omega) >0.$ In case of $\mathcal{B}(\omega_0)$ in $P$-tensor for a fixed $\omega_0$ with one additional assumption, SOL$(\mathcal{B}(\omega), p(\omega), \Omega)$ is nonempty. The limit $R_0$-property of SVTCP generalizes the concept of $R_0$-tensor in case of set-valued tensor complementarity problem. We show that the level boundedness of the merit function exists in case of the limit $R_0$-property of the set of tensors with some additional assumptions.

\section{Acknowledgment}
The author R. Deb is grateful to the Council of Scientific $\&$ Industrial Research (CSIR), India, for the financial support provided through the Junior Research Fellowship program.

\bibliographystyle{plain}
\bibliography{referencesall}

\end{document}